\documentclass[11pt]{amsart}
\usepackage[left=2cm,top=2cm,right=3cm,nofoot]{geometry}
\usepackage{amsmath,mathtools}%
\usepackage{amsfonts}%
\usepackage{amssymb}%
\usepackage{graphicx}
\usepackage{esint}
\usepackage{hyperref}
\usepackage{enumitem}
\usepackage{accents}
\usepackage{etoolbox}
\patchcmd{\subsection}{-.5em}{.5em}{}{}
\newtheorem{theorem}{Theorem}[section]
\theoremstyle{plain}

\newtheorem{lemma}[theorem]{Lemma}

\newtheorem{proposition}[theorem]{Proposition}
\newtheorem{remark}[theorem]{Remark}

\numberwithin{equation}{section}
\theoremstyle{plain}
\newtheorem{claim}{Claim}
\numberwithin{claim}{theorem} 
\usepackage{etoolbox}
\newcommand{\R}{\mathbb{R}}
\newcommand{\pa}{\partial}
\newcommand{\Da}{\Delta}
\newcommand{\ve}{\varepsilon}
\newcommand{\Ue}{U_{\ve,x}}
\begin{document}
\title[Solutions of the Yamabe equation]{Solutions of the  Yamabe equation by Lyapunov-Schmidt reduction}
\author{Jorge D\'avila}
\address{}
\email[Jorge D\'avila]{jorge.davila@cimat.mx}
\address{ CIMAT A.C., A.P. 402, 36000, Guanajuato. Gto., M\'exico.}

\author{Isidro H. Munive}
\address{Department of Mathematics, University Center of Exact Sciences and Engineering, University of Guadalajara, 44430 Guadalajara, Mexico}
\email[Isidro H. Munive]{isidro.munive@academicos.udg.mx}%

\begin{abstract} Given any closed Riemannian manifold $(M,g)$ we use the Lyapunov-Schmidt finite-dimensional reduction method and the classical Morse and Lusternick-Schnirelmann theories to
prove multiplicity results for positive solutions of a subcritical Yamabe type equation on $(M,g)$. If $(N,h)$ is a closed Riemannian
manifold of constant  positive scalar curvature we obtain multiplicity results for the Yamabe equation on the
Riemannian product $(M\times N , g + \ve^2 h )$, for $\ve >0$ small. For example, if $M$ is a closed Riemann 
surface of genus ${\bf g}$ and $(N,h) = (S^2 , g_0)$ is the round 2-sphere, we prove that for $\ve >0$ small enough
and a generic metric $g$ on $M$, 
the Yamabe equation on $(M\times S^2 , g + \ve^2 g_0 )$ has at least $2 + 2 {\bf g}$ solutions. 

\end{abstract}
\maketitle

\section{\textbf{Introduction}}

In  \cite{Yamabeart} H. Yamabe considered the following question: Let $(M, g)$ be a closed Riemannian manifold  of dimension $n\geq 3$. Is there a metric $h$   which is conformal  to $g$ and has constant  scalar curvature? If we express the conformal metric $h$ as $h=u^{\frac{4}{n-2}}g$ for a positive function $u$, the scalar curvature $s_h$ 
of $h$ is related to the scalar curvature of $g$ by 

$$
  -a_{n}\Delta_{g}u + s_gu = s_hu^{p_n -1},$$

\noindent
where  $\Delta_{g}$ is the Laplacian operator associated with the metric $g$, $a_{n}=\dfrac{4(n-1)}{(n-2)}$ and $p_n =\dfrac{2n}{n-2}.$ It follows that  the metric $h$ has  constant scalar curvature  $\lambda \in \R$ if and only if $u$ is a  positive solution of the {\it Yamabe equation}:
\begin{equation} \label{yamabeEquation}
  -a_{n}\Delta_{g}u + s_gu = \lambda u^{p_n -1}.
  \end{equation}

It is easy to check that Eq.  (\ref{yamabeEquation}) is the Euler-Lagrange equation of the {\it Yamabe functional},
$Y_g$, defined by:
\begin{equation} \label{FunctionalYamabe}
 Y_{g}(u)=\dfrac{\int\limits_{M} \Big(a_{n}\vert \nabla u\vert^2   +  s_gu^2 \Big)d\mu_{g}}{\Big(\int\limits_{M} u^{p_n} \ d\mu_g \Big )^{\frac{n-2}{n}}}= \dfrac{\int\limits_{M} \Big(a_{n}\vert \nabla u\vert^2   +  s_gu^2 \Big)d\mu_{g}}{\|u\|_{ p_n}^{2}}.
\end{equation}
If $\mathcal{E}$ denotes the normalized Hilbert-Einstein functional
\begin{equation*} 
 \mathcal{E}(g)=\dfrac{\int\limits_{M} s_gd\mu_{g}}{Vol(M,g)^{\frac{n-2}{n}}},
 \end{equation*}
it follows that  $Y_{g}(u)=\mathcal{E}(u^{\frac{4}{n-2}}g)$.

The Yamabe constant of $g$  is defined  as the infimum of the Yamabe functional $Y_{g}$ : 
\begin{equation}
Y(M,g)=\inf\limits_{u \in H^{1}(M) - \{ 0 \}} Y_{g}(u) .
\end{equation}

A minimizer for the Yamabe constant is therefore a solution of (\ref{yamabeEquation}) and, moreover,  from elliptic theory
this must be strictly positive and smooth. Yamabe presented a proof that a minimizer always exists, but his argument 
contained an error which was pointed out (and fixed under certain conditions) by N. Trudinger in \cite{t}. 
Later T. Aubin \cite{a1} and R. Schoen \cite{sch} completed the proof  that  for any metric $g$ the infimum of the  Yamabe functional is achieved. Therefore there  is always at least one (positive) solution to the Yamabe equation  (\ref{yamabeEquation}).   If $Y(M,g)\leq 0$ the solution is unique (up to homothecies). In the case of $Y(M, g) > 0$ uniqueness in general fails. The sphere $(S^{n}, g_{o})$ with  the curvature one metric is a first example of multiplicity of solutions. 

The case of the sphere is very special because it has a non-compact family of conformal transformations which induces a noncompact family of solutions to the Yamabe equation. By a result of M. Obata \cite{Obata}  each metric of constant scalar curvature which is conformal to the round metric on $S^{n}$ is obtained  as the pull-back of the round metric under a conformal diffeomorphism. Therefore, if $g_{o}$ is the round metric over $S^{n}$, every solution to (\ref{yamabeEquation}) is minimizing. But in general, for the positive case there will be non-minimizing
solutions. For instance, D. Pollack proved in \cite{Pollack}  that every conformal class with positive Yamabe constant can be $C^{0}$-approximated by a conformal class with an arbitrary number of (non-isometric) metrics of constant scalar curvature which are not
minimizers. Also,  S. Brendle in  \cite{Brendle} constructed smooth examples of Riemannian metrics with a non-compact family of 
non-minimizing solutions of the Yamabe equation.

Another important example was considered  by  R. Schoen in \cite{Schoen} (and also by O. Kobayashi in
\cite{Kobayashi}). In \cite{Schoen} Schoen worked with  the product metric on $S^{n-1}\times S^{1}({L})$
(the circle of radius $L$). He showed that  all solutions to (\ref{yamabeEquation}) are constant along the 
$(n-1)$-spheres and, therefore, the Yamabe equation reduces to an ordinary differential equation. By a careful analysis of  this equation,  Schoen proved that there are many non-mimizing solutions if $L$ is large.

Similar to the case of $S^{n-1}\times S^{1}({L})$,
particular interest arises in the study of products of the form $(M\times N , g + \delta h)$, where the constant $\delta >0$ goes
to 0 (or $\infty$). 
The Yamabe constants of  such Riemannian products were
studied in \cite{Akutagawa}. Multiplicity results for the Yamabe equation were obtained in \cite{Bettiol-Piccione, Lima-Piccione-Zedda, 
Lima-Piccione-Zedda2, Henry, Qinian_YanYan, Petean0} using
bifurcation theory, assuming that the scalar curvatures of $g$ and $h$ are constant.

In this  paper we   consider the case of Riemannian products were one of the scalar curvatures is not a constant.
Let $(M^n,g)$ be any closed Riemannian manifold and $(N^m,h)$ be a Riemannian manifold of constant positive scalar curvature. The function $u:M\rightarrow \R_{>0}$ is a solution of  the Yamabe   equation in $(W,g_{\varepsilon})=(M\times N,g+\varepsilon^2h)$ if it satisfies
\[
-a_{n+m}\Da_gu+\left(s_g+\varepsilon^{-2}s_h\right)u=u^{p_{m+n}-1}.
\]
 This is of
course equivalent to finding solutions of the equation
\begin{equation}
\label{Yamabe2}
-a_{n+m}\Da_gu+\left(s_g+\varepsilon^{-2}s_h\right)u=\varepsilon^{-2}s_hu^{p_{m+n}-1}.
\end{equation}
Moreover, we can normalize $h$ and assume that $s_h = a_{m+n}$. Then  Eq. (\ref{Yamabe2}) is equivalent to:
\begin{equation}
\label{yam-nor}
-\ve^2 \Da_gu+\left(\frac{s_g}{a_{m+n}}\varepsilon^{2}+1\right)u=u^{p_{m+n}-1}.
\end{equation}

We will find solutions of (\ref{yam-nor}) using the Lyapunov-Schmidt reduction technique, which was introduced 
  in  \cite{Bahari-Coron,Floer, Li}, for instance. The same technique was also used by  A. M. Micheletti and  A. Pistoia in
\cite{Micheletti-Pistoia} to study  the sub-critical equation equation $-\ve^2 \Delta_{g} u + u = u^{p-1}$ on
a Riemannian manifold. Here   we will use a similar approach. We now give  a brief description of this method and state the results we have obtained.

 Let $H_{\varepsilon} (M)$ be the Hilbert space $H^1_g(M)$ equipped with the inner product
\[
\langle u,v\rangle_{\ve}\doteq \frac{1}{\ve^n}\left(\ve^2\int_M \langle \nabla_gu , \nabla_gv \rangle \ d\mu_g+\int_Muv \ d\mu_g\right),
\]
and the induced  norm
\[
\|u\|^2_{\ve}\doteq  \frac{1}{\ve^n}\left(\ve^2\int_M|\nabla_gu|^2d\mu_g+\int_Mu^2d\mu_g\right).
\]
Consider the functional  $J_{\varepsilon} : H_{\varepsilon}  (M) \rightarrow \R$ given by

$$J_{\varepsilon} (u) = \varepsilon^{-n} \int_M \left( \frac{1}{2} \varepsilon^2 | \nabla u |^2  +  \frac{{\bf s}_g  \varepsilon^2 + a_{m+n} }{2a_{m+n}} u^2 -\frac{1}{p_{m+n}}  (u^+ )^{p_{m+n}} \right) d\mu_g  .$$ 

\noindent
where $u^+ =\max \{ u , 0 \}$. The critical points of the functional  $J_{\ve}$  are the positive solutions of Eq. (\ref{yam-nor}).
Let us consider the map 
$$S_{\ve} \doteq \nabla J_{\varepsilon} : H_{\ve}\rightarrow H_{\ve}.$$
 The Yamabe equation (\ref{yam-nor}) is  then 
equivalent to $S_{\ve}(u)  =0.$

Note that $p_{m+n}<p_n$. 
From now on we let $q \in (2 , p_n )$.  
There exists a unique (up to translation) positive finite-energy solution $U$  of the equation on $\R^n$

\begin{equation}
\label{limeq}
-\Delta U + U = U^{q-1}. 
\end{equation}

The function $U$ is radial (around some fixed point). We also consider the linear equation
\[
-\Delta\psi+\psi=(q-1)U^{q-2}\psi\quad \text{in $\R^n$}.
\]
It is well known that all solutions of above equation are the directional derivatives of $U$, i.e., the solutions are of the form
\[
\psi^v(z)\doteq \frac{\pa U}{\pa v}(z),  \text{ $v \in$ $\R^n$}.
\]
The function $U_{\varepsilon} (x) = U((1/\varepsilon ) x)$ is  a solution of
$$-\varepsilon^2 \Delta U_{\varepsilon} + U_{\varepsilon} = U_{\varepsilon}^{q-1}.$$
Similarly, we have that   $\psi_{\ve}^v (x) \doteq \psi^v ( (1/\varepsilon ) x)$ solves 
$$-\varepsilon^2 \Delta \psi_{\varepsilon} + \psi_{\varepsilon} = (q-1)U_{\varepsilon}^{q-2} \psi_{\varepsilon} .$$
Using   the exponential map $\exp_x :B(0,r) \rightarrow B_g (x,r)$, we define 
\[
U_{\ve,x}(y)\doteq 
\begin{cases}
U_{\ve}(\exp^{-1}_x(y))\chi_r(\exp^{-1}_x(y))& \text{if $y\in B_g(x,r)$},\\
0&\text{otherwise}.
\end{cases}
\]

We regard  $U_{\ve,x}$ as an approximate solution of Eq. (\ref{yam-nor}), and we will try to find an exact  solution of the form
$u\doteq U_{\ve,x}+\phi$,
where $\phi$  is a small perturbation.  For that we consider the following subspace of $H_{\ve} (M)$:
\[
K_{\ve,x}=  \Big\{W^v_{\ve,x} : v\in \R^n \Big\},
\]
where
\[
W^v_{\ve,x}(y)\doteq 
\begin{cases}
\psi^v_{\ve}(\exp^{-1}_x(y))\chi_r(\exp^{-1}_x(y))& \text{if $y\in B_g(x,r)$},\\
0&\text{otherwise}.
\end{cases}
\]

\noindent
$W^v_{\ve,x}$ is an approximate solution of the linearized equation $S_{\ve}' (U_{\ve,x} ) (v)=0$, and $K_{\ve,x}$ an
approximation to the kernel of $S_{\ve}' (U_{\ve,x} )$.

We are going to solve our equation modulo $K_{\ve,x}$ for $\phi$ in the orthogonal complement $K^{\perp}_{\ve,x}$ of $K_{\ve,x}$ in $H_{\ve}$. In other words, for $\ve >0$ small and $x\in M$,   we will find $\phi_{\ve,x}\in K^{\perp}_{\ve,x}$  such that
\[
\Pi^{\perp}_{\ve,x}\Big\{S_{\ve}\left(\Ue+\phi_{\ve,x}\right)\Big\} = 0,
\]
 where $\Pi^{\perp}_{\ve,x}:H_{\ve}\rightarrow K^{\perp}_{\ve,x}$ is the orthogonal projection. Hence, if for some $x_o\in M$ we have 
\[
\Pi_{\ve,x_o}\Big\{S_{\ve}\left(U_{\ve,x_o}+\phi_{\ve,x_o}\right)\Big\} = 0,
\]
with $\Pi_{\ve,x}:H_{\ve}\rightarrow K_{\ve,x}$ the orthogonal projection, then $U_{\ve,x_o}+\phi_{\ve,x_o}$ is a solution of Eq. (\ref{yam-nor}).
In this way,  the problem is reduced to a problem in finite dimensions. This is called the Lyapunov-Schmidt finite-dimensional reduction.

The following theorem is the key result of this paper:

\begin{theorem}
There exists $\ve_o >0$  such that for $\ve \in (0, \ve_o )$ and for any $x\in M$ there exists a unique $\phi_{\ve,x}\in K^{\perp}_{\ve,x}$ such that
\[
\Pi^{\perp}_{\ve,x}\Big\{S_{\ve}\left(\Ue+\phi_{\ve,x}\right)\Big\} = 0,
\]
and $\| \phi_{\ve,x} \|_{\ve} = O(\ve^2 )$. The map $x \in M  \mapsto J_{\ve} (U_{\ve ,x} + \phi_{\ve,x} )$ is $C^2$, and if 
$x_o$ is a critical point of this map  then  $U_{\ve ,x_o } + \phi_{\ve,x_o }$ is a positive solution of
equation (\ref{yam-nor}).

\end{theorem}

Let $F_{\ve} (x) =J_{\ve} (U_{\ve ,x} + \phi_{\ve,x} )$.
The  critical points of this  $C^2$ function on $M$ give positive solutions  of Eq. (\ref{yam-nor}). This allows to  apply 
the classical results about the number of critical points of functions on closed manifolds. 

The most direct application comes from Lusternik-Schnirelmann Theory. 
Recall that the Lusternik-Schnirelmann category of $M$, $Cat(M)$, is  the minimal integer $k$ such that 
$M$ can be covered by $k$ subsets, $M\subset M_{1}\cup M_{2}...\cup M_{k} $, with $M_{i}$ closed and contractible in $M$. The classical result of Lusternick-Schnirelmann theory  says that any $C^1$ function on a closed manifold
$M$ has at least $Cat(M)$ critical points. Therefore,  from Theorem 1.1 (and the discussion above) we can deduce the following result, which was proved by J. Petean in \cite{Petean1} with a different approach:

\begin{theorem}
  Let (M,g) be any closed Riemannian manifold and (N, h) be a Riemannian manifold of constant positive scalar curvature.
  There exist $\ve_{o}>0$ such that for  $0<\ve < \ve_{o}$   the Yamabe equation on the Riemannian product $(M\times N , g + \ve^2 h )$ has at least $Cat(M)$ solutions which depend only on $M$.
\end{theorem}

 In \cite{Petean1}  J. Petean proves  the existence of $Cat(M)$ low energy solutions and one
higher energy solution. The solutions provided in our theorem have low energy and they are close to the explicit 
approximate solutions. We also mention that C. Rey and M. Ruiz \cite{Rey_Ruiz} also applied the Lyapunov-Schmidt reduction technique to construct 
{\it multipeak} high-energy solutions
under certain conditions. 
These seem to be the only known  results when the scalar curvature of $g$ is not a constant.

Further applications can be obtained using Morse Theory. For that we have to  consider the asymptotic expansion of
$F_{\ve}$ in terms of $\ve$.  Similar expansions were considered when studying solutions of the equation  $-\ve^2 \Delta_{g} u + u = u^{p-1}$ on
a Riemannian manifold by A. M. Micheletti and A. Pistoia for instance in 
\cite{Micheletti-Pistoia}. Positive solutions of this equation are the critical points of the functional

$$J^0_{\varepsilon} (u) = \varepsilon^{-n} \int_M \left( \frac{\varepsilon^2}{2}  | \nabla u |^2  +  \frac{1 }{2} u^2 -\frac{1}{p_{m+n}}  (u^+ )^{p_{m+n}} \right) d\mu_g  .$$

Then A. M. Micheletti and A. Pistoia perform the Lyapunov -Schmidt reduction and define the map
$F^0_{\ve} (x) =J^0_{\ve} (U_{\ve ,x} + \phi_{\ve,x} )$ and prove in  \cite[Lemma 5.1]{Micheletti-Pistoia} that 
we have the following $C^1$-uniformly expansion:

$$F^0_{\ve} (x)= \alpha - \dfrac{\ve^{2}}{6}s_{g}(x)\int_{\mathbb{R}^{n}}\Big(\dfrac{U'(|z|)}{|z|}\Big)^{2}z_{1}^{4}dz   + o(\ve^{2}) ,$$

\noindent
where $U=U_{p_{m+n}}$ is the solution of equation (\ref{limeq}) with $q=p_{m+n}$, $U'$ means the derivative of $U$ in the radial direction, and 
$\alpha = \dfrac{1}{2}\|U\|_{H^{1}(\mathbb{R}^{n})}^{2}-\dfrac{1}{p}\|U\|^{p}_{L^{p}(\mathbb{R}^{n})}$.

There is an extra factor in the functional $J_{\ve}$ involving $s_g \ve^2$, which has an effect in the expansion
of the function $F_{\ve}$.
This was considered by C. Rey and M. Ruiz in \cite[Lemma 3.3]{Rey_Ruiz}. They obtain:
\begin{eqnarray*}
F_{\ve}(x)&=& \alpha - \dfrac{\ve^{2}}{6}s_{g}(x)\int_{\mathbb{R}^{n}}\Big(\dfrac{U'(|z|)}{|z|}\Big)^{2}z_{1}^{4}dz  + \dfrac{\ve^{2}}{2a_{m+n}}s_{g}(x)\int_{\mathbb{R}^{n}} U(z)^{2}dz + o(\ve^{2})\\
&=& \alpha +\dfrac{\beta_{m,n}}{2}\ve^{2}s_{g}(x) + o(\ve^{2})
\end{eqnarray*}
  which is $C^{1}-$uniformly with respect to $x$ when $\varepsilon$  tends to zero, where

$$\beta_{m,n}=\dfrac{1}{a_{m+n}}\int_{\mathbb{R}^{n}} U^{2}(z)dz -\dfrac{1}{3}\int_{\mathbb{R}^{n}}\Big(\dfrac{U'(|z|)}{|z|}\Big)^{2}z_{1}^{4}dz.$$

In \cite{Rey_Ruiz} it is also proved that $$\beta_{m,n}=\dfrac{1}{a_{m+n}}\int_{\mathbb{R}^{n}} U^{2}(z)dz -\dfrac{1}{n(n+2)}\int_{\mathbb{R}^{n}}|\nabla U(z)|^{2}|z|^{2}dz,$$ 

\noindent
and that numerical computations show that $\beta_{m,n} \neq 0$ if $m+n \leq 8$. It is difficult
to prove analitically that $\beta_{m,n} \neq 0$ in general but in Section 6 we will prove it in the case $m=n=2$. Assuming 
that $\beta_{m,n} \neq 0$ and that $x_0$ is a nondegenerate critical point of $s_g$ it is easy to prove, using the previous expansion, that for any $\delta >0$, if $\ve$ is small enough, then $F_{\ve}$ has a critical point in  $B(x_0 , \delta )$. It
was proved by A. M. Micheletti and A. Pistoia  in 
\cite{Micheletti-Pistoia2} that for a generic metric (on any closed manifold) all the critical points of its scalar curvature 
are nondegenerate, i. e. the scalar curvature is a Morse function on the manifold. 
We can  then apply Morse theory. Let $b_i (M) \doteq \dim (H_i (M, \R ))$ and $b(M)\doteq b_1 (M) + \dots +b_n (M)$. If
$f$ is a Morse function on $M$  then $f$ has at least
$b(M)$ critical points.
Therefore we obtain:

\begin{theorem}
    Let  $(N, h)$ be a closed Riemannian manifold of dimension $m$ of constant positive scalar curvature. Let $M$ be a closed manifold of dimension $n$. Assume that $\beta_{m,n}  \neq 0$. For a generic Riemannian metric $g$ on $M$  there exist $\ve_{o}>0$ such that if  $0<\ve < \ve_{o}$ the Yamabe equation on the Riemannian product $(M\times N , g + \ve^2 h )$ has at least 
$b(M)$ positive solutions.
  \end{theorem}

Using that $\beta_{2,2} \neq 0$ we have:

\begin{theorem} Let $g_0$ be the round metric on the sphere $S^2$. Let $M$ be a closed manifold of dimension $2$.  For a generic Riemannian metric $g$ on $M$  there exist $\ve_{o}>0$ such that if  $0<\ve < \ve_{o}$ the Yamabe equation on the Riemannian product $(M\times S^2 , g + \ve^2 g_0 )$ has at least 
$b(M)$ positive solutions.
  \end{theorem}

In case the scalar curvature of $g$ is constant the expansion of $F_{\ve}$ up to order $\ve^2$ is constant and to
obtain critical points one would need to consider higher order expansions. For the equation 
$-\ve^2 \Delta_{g} u + u = u^{p-1}$ such an expasion was considered for instance by S. Deng, Z. Khemiri
and F. Mahmoudi in \cite{DKM}.

\vspace{.5cm}

In Sections 2 and  3 we  will discuss some preliminary results about the Lyapunov-Schmidt reduction technique and  prove some delicate estimates
involving the approximate solutions. In Section 4 we  prove the existence of the appropriate perturbation functions $\phi_{x,\ve}$, see Proposition 4.2. In Section  5 we  complete the proof of Theorem 1.1. Finally in Section 6 we will
prove that $\beta_{2,2} \neq 0$.

\section{\textbf{Preliminaries}}

\subsection{The limiting equation and its solution  on $\R^n$}

Let $2<q<p_n$ (where if $n=2$ then $p_n =  \infty $). It is well known that  there exists a unique (up to translation) positive finite-energy solution $U$  of the equation 

$$-\Delta U + U = U^{q-1},\quad \text{ in $\R^n$}. $$
The function $U$ is radial (around some chosen point) and it is exponentially decreasing at infinity
(see for instance \cite{Gidas}):
$$|U (x) |  \leq C e^{-c| x | } \quad\text{and}\quad | \nabla U (x) | \leq C  e^{-c| x | }.$$

Consider  the functional $E: H^1 (\R^n ) \rightarrow \R$,

$$E(f)= \int_{\R^n} (1/2) | \nabla f  |^2 + (1/2) f^2 -(1/q) (f^+)^q  \ dx ,$$ 

\noindent 
where $f^+ (x):=\max \{ f(x), 0\}$. Note that $U$ is a critical point of $E$.

For any $\varepsilon >0$ let 
$$E_{\varepsilon} (f)=\varepsilon^{-n}  \int_{\R^n} (\varepsilon^2 /2) |\nabla f |^2 + (1/2) f^2 -(1/q) (f^+)^q  \ dx .$$ 
The function  $U_{\varepsilon} (x) \doteq U((1/\varepsilon ) x)$  is a critical point of $E_{\varepsilon}$, i.e. a solution of
\begin{equation}
\label{limequatione}
-\varepsilon^2 \Delta U_{\varepsilon} + U_{\varepsilon} = U_{\varepsilon}^{q-1}.
\end{equation}

\noindent

Now, let us consider the linear equation
\begin{equation}
\label{lineareq}
-\Delta\psi+\psi=(q-1)U^{q-2}\psi\quad \text{in $\R^n$}.
\end{equation}
It is well known that all solutions of Eq. (\ref{lineareq}) are the directional derivatives of $U$, i.e. the solutions are of the form
\[
\psi^v(z)\doteq \frac{\pa U}{\pa v}(z),  \text{ $v \in$ $\R^n$}.
\]
In particular,
set $\psi^i \doteq \psi^{e_i}$. Since $U$ is radial, we have that the set $\{\psi^1,\ldots,\psi^n\}$ is orthogonal in $H^1(\R^n)$, i.e.
\begin{equation}
  \label{ortoline}
\int_{\R^n}\Big\{ \langle \nabla\psi^i , \nabla\psi^j \rangle +\psi^i(z)\psi^j(z)\Big\}dz=0, \quad \text{for $i\neq j$}.
\end{equation}
For more details  see for instance \cite{Gidas, Kwong, Wei-Winter}.

\subsection{The setting on a Riemannian manifold}

 Let $H_{\varepsilon}$ be the Hilbert space $H^1_g(M)$ equipped with the inner product
\[
\langle u,v\rangle_{\ve}\doteq \frac{1}{\ve^n}\left(\ve^2\int_M \langle \nabla_gu , \nabla_gv \rangle d\mu_g+\int_Muvd\mu_g\right),
\]
and the induced  norm
\[
\|u\|^2_{\ve}\doteq  \frac{1}{\ve^n}\left(\ve^2\int_M|\nabla_gu|^2d\mu_g+\int_Mu^2d\mu_g\right).
\]
Let $L^q_{\ve}$  be the Banach space $L^q_g(M)$ with the norm
\[
|u|_{q,\ve}\doteq \left(\frac{1}{\ve^n}\int_M|u|^qd\mu_g\right)^{1/q}, \quad u\in L^q_g(M).
\]
The standard norm in $L^q_g(M)$ will be denoted from now on by
\[
|u|_{q}\doteq \left(\int_M|u|^qd\mu_g\right)^{1/q},\quad u\in L^q_g(M),
\]

\begin{remark} For $u\in H^1 (\R^n )$ we let $u_{\ve} (x) =u(\ve^{-1} x)$. For any $\ve >0$ we have

\begin{equation}
\| u_{\ve} \|_{\ve} =\| u \|_{H^1} 
\end{equation}
and
\begin{equation}
| u_{\ve}  |_{q,\ve } =   | u |_{q} .
\end{equation}
\end{remark}

\begin{remark}
\label{emb}
For $q\in (2,p_n )$ if $n\geq 3$ or $q > 2$ if $n=2$,  the embedding $i_{\ve} : H_{\ve} \hookrightarrow L^q_{\ve}$ is a continuous map. Moreover, one can easily check that there exists a  constant $c$  independent of $\ve$ such that 
\[
|i_{\ve} (u) |_{q,\ve}\leq c\|u\|_{\ve},\quad \text{for any $u\in H_{\ve}$}.
\]
Let $q'=\frac{q}{q-1}$, so that $\frac{1}{q} + \frac{1}{q'} =1$. Then, there exists a continuous operator $i^*_{\ve}:L^{q'}_{\ve}\rightarrow H_{\ve}$, called the adjoint of $i_{\ve}$,  such that 
\[
\langle i^*_{\ve}(v),\varphi\rangle_{\ve}=\langle v,i_{\ve}\left(\varphi\right)\rangle\doteq \frac{1}{\ve^n}\int_Mv\cdot i_{\ve}\left(\varphi\right),\quad \forall\ v\in L^{q'}_{\ve}\ \text{and}\ \forall\ \varphi \in H_{\ve}
\]
In order to see this, we notice that for  $v\in L^{q'}_{\ve}$, the map $\mathfrak{F}_v:H_{\ve}\rightarrow\R$, given by
\[
\mathfrak{F}_v\left(\varphi\right)=\langle v,i_{\ve}\left(\varphi\right)\rangle,\quad \varphi \in H_{\ve},
\]
is a continuous functional by the compact embedding $i_{\ve}:H_{\ve}\hookrightarrow L^q_{\ve}$.  By the Riesz representation theorem, there exists $u_{v}\in H_{\ve}$ such that
\begin{equation}
\label{adj}
\mathfrak{F}_v(\varphi)=\langle u_v,\varphi\rangle_{\ve}, \quad \forall\,\varphi\in H_{\ve}.
\end{equation}
Therefore, $i^{\ast}(v)=u_v$. Finally, observe that 
\begin{equation}
  \label{bas1}
\|i^*_{\ve}(v)\|_{\ve}\leq c|v|_{q',\ve}, \quad \text{for any $v\in L^{q'}_{\ve}$},
\end{equation}
where the constant $c>0$ does not depend on $\ve>0$. 

\end{remark}

Recall that if  $v\in L^{q'}_{\ve}$,  then a function $u$ is a solution of
\begin{equation}
\label{eqv}
\begin{cases}
-\ve^2\Delta_gu+u=\  v \quad\text{in } M, \\
u\ \in\  H^1_g(M),
\end{cases}
\end{equation}
if and only if $u\in H^1_g(M)$, and it satisfies
\[
  \frac{1}{\ve^n}\left(\ve^2\int_M \langle \nabla_gu , \nabla_g \varphi  \rangle d\mu_g+\int_Mu\varphi \ d\mu_g\right)= \frac{1}{\ve^n}\int_Mv \cdot i_{\ve}\left(\varphi\right)  d\mu_g ,\quad\forall\ \varphi  \in H_{\ve}.
\]
If we define $u\doteq  i^*_{\ve}(v)$, with $v\in L^{q'}_{\ve}$, then $u$ is a solution of  (\ref{eqv}).  
This implies that  if $v \in C^k (M)$ then $u\in C^{k+2} (M)$.

Now, let $u\in H_{\ve}$, then 
\[
\frac{1}{\ve^n}\int_M |(u^+)^{q-1}|^{q'}d\mu_g\leq \frac{1}{\ve^n}\int_M|u|^{q}d\mu_g= |u|^q_{q,\ve}.
\]
Moreover, by Jensen's inequality
\begin{equation}
  \label{tro}
\Big|\frac{s_g(x)}{a_{m+n}}\varepsilon^{2}u\Big|_{q',\ve}\leq c_o \ve^{2+\frac{n}{q}-\frac{n}{q'}}|u|_{q,\ve},
\end{equation}
where $c_o >0$ depends only on $M$. It is easy to see that
\[
2+\frac{n}{q}-\frac{n}{q'}>0, \quad \text{since}\quad 2<q<\frac{2n}{n-2}.
\]
Now, we set $q\doteq p_{m+n}$. It follows that if $u \in H_{\ve}$, then
\[
F(u)\doteq (u^+(x))^{p_{m+n} -1} -\frac{s_g(x)}{a_{m+n}}\varepsilon^{2}u(x) \in L^{p_{m+n}'}_{\ve}.
\]

We   define the operator $S_{\ve} : H_{\ve} \rightarrow H_{\ve}$ by 
\begin{equation}
\label{yameq}
S_{\ve} (u)= u- i^*_{\ve}\left(F(u)\right).
\end{equation}
By the Remark \ref{emb}, $S_{\ve} (u) = \nabla J_{\ve} (u)$, where, as in the Introduction,
$$J_{\varepsilon} (u) = \varepsilon^{-n} \int_M \left( \frac{1}{2} \varepsilon^2 | \nabla u |^2  +  \frac{{ s}_g  \varepsilon^2 + a_{m+n} }{2a_{m+n}} u^2 -\frac{1}{p_{m+n}}  (u^+ )^{p_{m+n}} \right) d\mu_g  .$$ 
In particular,  $S_{\ve} (u) =0$  if and only if $u$ is a critical point of the functional 
 $J_{\varepsilon}$.

Note also that  
\begin{equation}
\label{lineary}
S'_{\ve}(u)\varphi=\varphi-\ i^*_{\ve}\left( (p_{m+n} -1) (u^+)^{p_{m+n} -2}\varphi-\frac{s_g(x)}{a_{m+n}}\varepsilon^{2}\varphi\right), \quad\varphi\ \in\  H_{\ve}(M).
\end{equation}

\section{\textbf{Approximate solutions}}

Let $U$ be the solution of Eq. (\ref{limeq}) where $q= p_{m+n}$. For simplicity we will use $p$ to denote $p_{n+m}$.  Let
\begin{equation}
\label{Ux}
U_{\ve,x}(y)\doteq 
\begin{cases}
U_{\ve}(\exp^{-1}_x(y))\chi_r(\exp^{-1}_x(y)),& \text{if $y\in B_g(x,r)$},\\
0,&\text{otherwise}.
\end{cases}
\end{equation}

Since $U_{\ve}$ solves (\ref{limequatione}), we consider  $U_{\ve,x}$ as an approximate solution 
of Eq. (\ref{yam-nor}). In this section we will prove some estimates related to $U_{\ve,x}$. Similar
estimates have been obtained before, see for instance in  \cite{Micheletti-Pistoia}. We sketch the proofs of
the estimates for completeness and to point out the  necessary adjustments to handle the extra
term $\frac{s_g\varepsilon^{2}}{a_{m+n}}u$  in Eq. (\ref{yam-nor}).

The function $U_{\ve,x}$ is an approximate solution in the following sense.

\begin{lemma}
\label{bRs}
There exists an $\ve_o>0$ and $C>0$ such that for every $x\in M$ and every $\ve \in (0,\ve_o)$ we have
\[
\|S_{\ve} (U_{\ve,x} ) \|_{\ve}\leq C\ve^2.
\]
\end{lemma}

\begin{proof} Observe \[\|S_{\ve} (U_{\ve,x} ) \|_{\ve} =\sup_{\|v\|_{\ve}=1}  \langle S_{\ve} (U_{\ve,x} ) , v \rangle _{\ve} .\]

Now $$ \langle S_{\ve} (U_{\ve,x} ) , v \rangle _{\ve} =\dfrac{1}{\ve^{n}}\int_{M}\Big[\ve^{2} \langle \nabla U_{\ve,x} ,  \nabla v \rangle  +  \left( 1 + \dfrac{s_g\varepsilon^{2}}{a_{m+n}}\right) U_{\ve,x} v - U_{\ve,x}^{p-1}v\Big]d\mu_g$$

$$=\dfrac{1}{\ve^{n}}\int_{M} \Big(-\ve^{2}\Delta U_{\ve,x}   +U_{\ve,x} - U_{\ve,x}^{p-1}\Big)v \ d\mu_{g} +\dfrac{1}{\ve^{n}}\int_{M} \dfrac{s_g\varepsilon^{2}}{a_{m+n} } U_{\ve,x} v  \ d\mu_g.$$
\\ 
On one hand
\begin{eqnarray}
\label{AA}
  \dfrac{\ve^{2}}{\ve^{n}}\bigg\rvert \int_{M} \dfrac{s_g}{a_{m+n}} U_{\ve,x} v d\mu_g\bigg\rvert
&\leq& C_1 \dfrac{\ve^{2}}{\ve^{n}}\int_{M} \rvert U_{\ve,x} v \rvert d\mu_{g} \\
&=& C_1 \ve^{2}|U_{\ve,x}|_{p', \ve}|v|_{\ve,p}\leq C_2 \  \ve^{2}|U_{\ve,x}|_{p', \ve},\nonumber
\end{eqnarray}

\noindent
using H\"older's inequality and Remark 2.2.
It follows from the exponential decay of $U$ and change of variables, as in Remark 2.1,   that  $\lim\limits_{\ve\rightarrow 0}|U_{\ve,x}|_{p', \ve}^{p'} =|U|_{p'}^{p'} < \infty .$ Therefore there exists $C>0$ such that 
\[
\bigg\rvert\frac{1}{\ve^{n}} \int_{M}\dfrac{s_g\varepsilon^{2}}{a_{m+n}} U_{\ve,x} v d\mu_g\bigg\rvert  \leq C\ve^{2} .
  \]
On the other hand,  we have by the embedding that
\begin{eqnarray*}
 \bigg\rvert \dfrac{1}{\ve^{n}}\int_{M} \Big(-\ve^{2}\Delta U_{\ve,x}   +U_{\ve,x} - U_{\ve,x}^{p-1}\Big)vd\mu_{g} \bigg\rvert &\leq& | -\ve^{2}\Delta U_{\ve,x}   +U_{\ve,x} - U_{\ve,x}^{p-1}|_{p',\varepsilon}|v|_{p,\varepsilon}\\
&\leq & c | -\ve^{2}\Delta U_{\ve,x}   +U_{\ve,x} - U_{\ve,x}^{p-1}|_{p', \ve}.
\end{eqnarray*}

From the proof of Lemma 3.3  in \cite{Micheletti-Pistoia}, we have   that there is positive constant $C$ and $\ve_o>0$ such that  for all $x\in M$ and $\ve\in (0,\ve_o)$ it holds,
\begin{equation} 
\Big| -\ve^{2}\Delta U_{\ve,x}   +U_{\ve,x} - U_{\ve,x}^{p-1}\Big|_{p', \ve}\leq  C \ve^{2} .
\end{equation}
This completes the proof of the lemma.

\end{proof}

We consider now the kernel of the linearized equation at the approximate
solution,  $\{ v \in H^1 (M) : S'_{\ve}(\Ue) (v) =0 \} $. In order to have information about the  kernel we consider  $\ve>0$, $x\in M$, and
pick an orthonormal basis of $T_x M$ to identified it with $\R^n$. Using  normal coordinates we define the following
subspace of $H^1 (M)$:
\[
K_{\ve,x}=  \Big\{W^v_{\ve,x} : v\in \R^n \Big\},
\]
where
\begin{equation}
\label{Wi}
W^v_{\ve,x}(y)\doteq 
\begin{cases}
\psi^v_{\ve}(\exp^{-1}_x(y))\chi_r(\exp^{-1}_x(y))& \text{if $y\in B_g(x,r)$},\\
0&\text{otherwise},
\end{cases}
\end{equation}
with $\psi^v_{\ve}(z)=\psi^v(\frac{z}{\ve})$ (as in the Introduction). Note that $W^v_{\ve,x}$ depends on the choice of the orthonormal basis
but the space itself $K_{\ve,x}$ does not. We will also set $W^i_{\ve,x} \doteq W^{e_i}_{\ve,x}$.

It is easy to see from (\ref{ortoline}) and Remark 2.1 that 
\begin{equation}\label{A}
 \langle W^i_{\ve,x},W^i_{\ve,x}\rangle_{\ve}\rightarrow C, \quad\langle W^i_{\ve,x},W^j_{\ve,x}\rangle_{\ve}\rightarrow 0\quad\text{if $i\neq j$}, \quad\text{as $\ve\rightarrow 0$},
\end{equation}
where the constant $C=\int_{\mathbb{R}^{n}}( \langle \nabla \psi^{i} , \nabla \psi^{i} \rangle + \psi^{i}\psi^{i})dx >0$ is independent of $i\in \{1,\ldots,n\}$ and  $x\in M$.

One can also show the following (details can be found in  Lemma 6.1  and Lemma 6.2  in \cite{Micheletti-Pistoia}). 
\begin{proposition}
\label{wii}
We have that
\begin{equation}\label{end1}
\lim_{\ve \rightarrow 0} \ \ve^2 \  \Big\| \frac{\partial}{\partial v}  W_{\ve ,x }^v \Big\|_{\ve}  \  = 0,
\end{equation}
and 
\begin{equation}\label{end2}
\lim_{\ve \rightarrow 0} \ve 
\Big\langle \frac{\partial}{\partial v}  (U_{\varepsilon ,x}  ) , W_{\ve ,x }^v  \Big\rangle_{\ve} = \langle \psi^v , \psi^v \rangle_{H^1} >0.
\end{equation}

\end{proposition}

The function $W^v_{\ve,x}$ is an approximate solution of the linearized equation in  the following sense.

\begin{lemma}
\label{bR}
For any $v\in \R^n$ there exists an $\ve_o>0$ and $C>0$ such that for every $x\in M$ and all $\ve \in (0,\ve_o)$ we have
\[
\|S'_{\ve} (U_{\ve,x} )  (W^v_{\ve ,x} ) \|_{\ve}\leq C\ve^2 \| v \| .
\]
\end{lemma}

\begin{proof} It is enough to consider the case $v =e_i $.
We have \[\|S'_{\ve} (U_{\ve,x} )  (W^i_{\ve ,x} ) \|_{\ve} =\sup_{\|w\|_{\ve}=1} \langle S'_{\ve} (U_{\ve,x})(W^i_{\ve ,x}),w
\rangle_{\ve}. \]
Now, we have that  
\begin{eqnarray*}
 \langle S'_{\ve} (U_{\ve,x})(W^i_{\ve ,x}),w
\rangle_{\ve}&=&\dfrac{1}{\ve^{n}}\int_{M}\Big[\ve^{2} \langle \nabla W^i_{\ve ,x} ,  \nabla w \rangle +  \left( 1 + \dfrac{s_g\varepsilon^{2}}{a_{m+n} }\right) W^i_{\ve ,x} \,w - (p-1 )(U_{\ve,x})^{p-2} W^i_{\ve ,x}\, w\Big]d\mu_g\\
&=&\dfrac{1}{\ve^{n}}\int_{M} \Big(-\ve^{2}\Delta W^i_{\ve ,x}  +W^i_{\ve ,x} - (p-1)(U_{\ve,x})^{p-2} W^i_{\ve ,x}\Big)w\, d\mu_{g}\\
&& +\dfrac{1}{\ve^{n}}\int_{M} \dfrac{s_g\varepsilon^{2}}{a_{m+n}} W^i_{\ve ,x}\, w\, d\mu_g.\\
\end{eqnarray*}
Observe that
\begin{eqnarray*}
  \dfrac{\ve^{2}}{\ve^{n}}\bigg\rvert \int_{M} \dfrac{s_g}{a_{m+n}} W^i_{\ve ,x} w d\mu_g\bigg\rvert&
 \leq & C \dfrac{\ve^{2}}{\ve^{n}}\int_{M} \rvert W^i_{\ve ,x} w \rvert d\mu_{g} \\
 & \leq & C\ve^{2}|W^i_{\ve ,x}|_{p^{'},\varepsilon}  |w|_{p,\varepsilon} 
  \leq   C\ve^{2}|W^i_{\ve ,x}|_{p', \ve},
  \end{eqnarray*}
by a similar argument as in (\ref{AA}).

It follows form the exponential decay of $\psi^i$ and change of variables that  $\lim_{\ve \rightarrow 0} |W^i_{\ve ,x}|_{p',\ve } = 
| \psi^i |_{p'}$. We conclude that 

\begin{equation}\label{AAA}
  \dfrac{\ve^{2}}{\ve^{n}}\bigg\rvert \int_{M} \dfrac{s_g}{a_{m+n}} W^i_{\ve ,x} w  \ d\mu_g\bigg\rvert \leq \overline{C} \ve^2 .
\end{equation}

Moreover, by Remark \ref{emb} we have
\begin{eqnarray*}
 &&\bigg\rvert \dfrac{1}{\ve^{n}}\int_{M} \Big(-\ve^{2}\Delta W^i_{\ve ,x}  +W^i_{\ve ,x} - (p-1)(U_{\ve,x})^{p-2} W^i_{\ve ,x}\Big)wd\mu_{g} \bigg\rvert \\
&=& \Big| -\ve^{2}\Delta W^i_{\ve ,x}   +W^i_{\ve ,x} - (p-1)(U_{\ve,x})^{p-2} W^i_{\ve ,x}\Big|_{p',\varepsilon}|w|_{p,\varepsilon}\\ 
&\leq & \Big| -\ve^{2}\Delta W^i_{\ve ,x}   +W^i_{\ve ,x} - (p-1)(U_{\ve,x})^{p-2} W^i_{\ve ,x}\Big|_{p',\varepsilon}\|w\|_{\ve} \\
&=&| -\ve^{2}\Delta W^i_{\ve ,x}   +W^i_{\ve ,x} - (p-1)(U_{\ve,x})^{p-2} W^i_{\ve ,x}|_{p', \ve}.
\end{eqnarray*}
It is shown in Lemma 5.2 of \cite{Micheletti-Pistoia}  that
\begin{equation}
\label{wterm}
\Big| -\ve^{2}\Delta W^i_{\ve ,x}   +W^i_{\ve ,x} - (p-1)(U_{\ve,x})^{p -2}W^i_{\ve ,x}\Big|_{p', \ve}
 \leq C\ve^2 ,
\end{equation}
\noindent
Estimate (\ref{wterm}) together with (\ref{AAA}) finishes the proof of the lemma.

  \end{proof}

We now solve $S_{\ve} (u)=0$  modulo $K_{\ve,x}$. We consider   the orthogonal complement $K^{\perp}_{\ve,x}$ of $K_{\ve,x}$ in $H_{\ve}$ and we  find $\phi_{\ve,x}\in K^{\perp}_{\ve,x}$ such that
\begin{equation}
\label{perpeq}
\Pi^{\perp}_{\ve,x}\Big\{S_{\ve}\left(\Ue+\phi_{\ve,x}\right)\Big\} = 0,
\end{equation}

\noindent
where $\Pi^{\perp}_{\ve,x} : H_{\ve} \rightarrow K^{\perp}_{\ve,x}$ is the orthogonal projection.
In the next section  we will show that  there exists $\ve_o=\ve_o(M)>0$, such that for every $x\in M$ and $\ve\in (0,\ve_o)$, there is a unique $\phi_{\ve,x}\in K^{\perp}_{\ve,x}$   that solves Eq. (\ref{perpeq}). It will remain then to find points $x\in M$ for which
\begin{equation}
\label{finiteeq}
\Pi_{\ve,x}\Big\{S_{\ve}\left(\Ue+\phi_{\ve,x}\right)\Big\} = 0 ,
\end{equation}
where $\Pi_{\ve,x} : H_{\ve} \rightarrow K_{\ve,x}$ is the orthogonal projection.

\section{\textbf{The finite-dimensional reduction}}

This section is devoted to solve Eq. (\ref{perpeq}). For $x\in M$ and $\ve >0$ we consider the linear operator $L_{\ve,x}:K^{\perp}_{\ve,x}\rightarrow K^{\perp}_{\ve,x}$ defined by
\[
L_{\ve,x}(\phi)\doteq \Pi^{\perp}_{\ve,x}\Big\{S'(U_{\ve,x})\phi\Big\},
\]
where by (\ref{lineary}) 
\[
S'(U_{\ve,x})\phi=\phi-i^*_{\ve}\Big[(p-1)(\Ue)^{p_{m+n}-2} \phi-\varepsilon^{2}\frac{s_g}{a_{m+n}}\phi\Big].
\]

In the following proposition we show that the bounded operator $L_{\ve,x}$ satisfies a coercivity estimate for $\ve>0$ small enough, uniformly on $M$. From this result it follows  the invertibility of $L_{\ve,x}$ for $\ve>0$ small.

\begin{proposition}
\label{invl}
There exists $\ve_o>0$ and $c>0$ such that for any point $x\in M$ and for any $\ve\in (0,\ve_o)$
\[
\|L_{\ve,x}(\phi)\|_{\ve}\geq c\|\phi\|_{\ve} \quad \text{for all $\phi\in K^{\perp}_{\ve,x}$}.
\]
\end{proposition}
\begin{proof}
Assume the proposition is not true. Then there exists a sequence of positive numbers $\ve_i $, with
$\lim_{i\rightarrow \infty} \ve_i =0$, and sequences $\{x_i\} \subset M$,  $\{\phi_i\} \subset K^{\perp}_{\varepsilon_i ,x_i} $ with
$\| \phi_i \|_{\ve_i} =1$, such that $\| L_{\varepsilon_i , x_i} (\phi_i ) \|_{\ve_i} \rightarrow 0.$ Moreover, since $M$ is compact we can assume that there exists $x\in M$ such that $x_i \rightarrow x$. 
\begin{claim}
\label{xi0}
Let $\omega_i\doteq L_{\varepsilon_i , x_i} (\phi_i ) $ and set
\begin{equation}
\label{xieq}
\xi_{i}\doteq S'_{\ve_i}(U_{\ve_i,x_i})\phi_i-\omega_i\in K_{\ve_i,x_i}.
\end{equation}
Then, 
\[
\|\xi_{i}\|_{\ve_i}\rightarrow 0,\quad \text{as $i\rightarrow \infty$}.
\]
\end{claim}
\begin{proof}[Proof of Claim \ref{xi0}]
To prove the claim note that for any $v \in \R^n$, 
\begin{eqnarray*}
\langle \xi_i, W^v_{\ve_i,x_i}\rangle_{\ve_i} 
=\langle S'_{\ve_i}(U_{\ve_i,x_i})\phi_i, W^v_{\ve_i,x_i}\rangle_{\ve_i}=\langle \phi_i, S'_{\ve_i}(U_{\ve_i,x_i}) (W^v_{\ve_i,x_i} ) \rangle_{\ve_i}.
\end{eqnarray*}
The claim then follows from Lemma \ref{bR}. 
\end{proof}

Now, we have
\begin{equation}
\label{uif}
u_i\doteq\phi_i-\omega_i-\xi_i = \phi_i- S'_{\ve_i}(U_{\ve_i,x_i})\phi_i =  i^*_{\ve_i}\left((p-1)(U_{\ve_i,x_i} )^{p-2} \phi_i -\frac{s_g(x)}{a_{m+n}}\varepsilon_i^{2}\phi_i \right),
\end{equation}
by (\ref{lineary}). It follows from Claim \ref{xi0} that 
\begin{equation}
\| u_i \|_{\ve_i}  \rightarrow 1.
\end{equation}

From Remark 2.2 and Eq. (\ref{uif}), $u_i$ solves 
\begin{equation}
\label{xieq1}
-\ve_i^2\Delta_gu_i +u_i =\  (p-1)(U_{\ve_i,x_i}  )^{p-2} \phi_i -\frac{s_g(x)}{a_{m+n}}\varepsilon_i^{2}\phi_i .
\end{equation}
Let 
\[
v_i\doteq i^*_{\ve_i}\left((p-1)(U_{\ve_i,x_i} )^{p-2} \phi_i \right) = u_i +  i^*_{\ve_i} \left( \frac{s_g(x)}{a_{m+n}}\varepsilon_i^{2}\phi_i \right) .
\]
Then $v_i$ is supported in $B(x_i ,r)$ and 
\begin{equation}
\label{ui01}
\|v_i\|_{\ve_i}\rightarrow 1 \ \ \ \ \ ,  \ \ \   \| v_i -\phi_i \|_{\ve_i} \rightarrow 0 .
\end{equation}
Moreover, it solves 
\begin{equation}
\label{xieq2}
-\ve_i^2\Delta_gv_i +v_i =\  (p-1)(U_{\ve_i,x_i}  )^{p-2} \phi_i   .
\end{equation}

\begin{claim}
\label{uiw0c}
Let
\[
\widetilde{v}_i(y)\doteq v_i\left(\exp_{x_i}\left(\ve_i y\right)\right), \quad y\in B\left(0,r/\ve_i\right) \subset \R^n .
\]
Then,
\begin{equation}
\label{uiw0}
\widetilde{v}_i\rightarrow 0\quad \text{weakly in $H^1(\R^n)$ and strongly in $L^q_{loc}(\R^n)$},
\end{equation}
for any $q\in(2, p_n )$ if $n\geq 3$ or $q > 2$ if n=2.
\end{claim}

\begin{proof}[Proof of Claim \ref{uiw0c}]
Let $\widetilde{v}_{i_{\ve_i}} (y) = \widetilde{v}_i (\ve_i^{-1} y)= v_i \left( \exp_{x_i} ( y ) \right) $. Observe that
\begin{equation}
\label{uibd}
\|\widetilde{v}_i\|_{H^1( \R^n )}  = \|\widetilde{v}_{i_{\ve_i}} \|_{H_{\ve_i}( \R^n )}   \leq C\|v_i\|_{\ve_i}\leq C, \quad \text{for all $i\in\mathbb{N}$}.
\end{equation}
Therefore, by taking  a subsequence we can assume that there exists $ \widetilde{v} \in H^1 (\R^n )$ such that  $\widetilde{v}_i\rightarrow \widetilde{v}$ weakly in $H^1(\R^n )$, and strongly in $L^q_{loc}(\R^n)$ for any $q\in (2, p_n )$ if $n\geq 3$ or $q > 2$ if $n=2$.

Now, observe that by Claim \ref{xi0}  for $j=1,\ldots,n$,
\begin{equation}
\label{Wui1}
\langle W^j_{\ve_i,x_i}, v_i\rangle_{\ve_i}=\langle W^j_{\ve_i,x_i}, u_i\rangle_{\ve_i}  +o(\ve_i )=
-\langle W^j_{\ve_i,x_i}, \xi_i\rangle_{\ve_i} +o(\ve_i )\rightarrow0,\quad \text{as $i\rightarrow \infty$},
\end{equation}
and (by change of variables and the exponential decay of $\psi^j$)
\begin{equation}
\label{Wui2}
\langle W^j_{\ve_i,x_i}, v_i\rangle_{\ve_i}\rightarrow \int_{\R^n}\left(\nabla\psi^j\nabla \widetilde{v}+\psi^j\widetilde{v}\right)dy, \quad\text{as $i\rightarrow \infty$}.
\end{equation}

We have from (\ref{ui01}) and (\ref{xieq2}) that $\widetilde{v}$ solves
\begin{equation}
\label{solwu}
-\Delta\widetilde{v}+\widetilde{v}=(p-1)(U)^{p-2} \widetilde{v}\quad \text{in $\R^n$}.
\end{equation}
Therefore, $\widetilde{v}\in \mathrm{span}\{\psi^1,\ldots,\psi^n\}$. From Eq.'s (\ref{Wui1}) and (\ref{Wui2}), we have that $\widetilde{v}$ is orthogonal to $\{\psi^1,\ldots,\psi^n\}$, hence $\widetilde{v}\equiv 0$, and the claim follows. 
\end{proof}

Multiplying  Eq. \ref{xieq2} by $v_i\in H_{\ve}$, we obtain from (\ref{ui01})
\begin{eqnarray}
\label{ui}
\|v_i\|^2_{\ve_i}&=&\frac{1}{\ve^n_i}\int_M\Big\{(p-1)(U_{\ve_i,x_i})^{p-2} \Big\} v_i \  \phi_i \rightarrow 1
\end{eqnarray}
But, by  Claim \ref{uiw0c} we have
\begin{eqnarray}
\frac{1}{\ve^n_i}\int_M\Big\{(p-1)(U_{\ve_i,x_i})^{p-2} \Big\} v_i \  \phi_i \rightarrow \int_{\R^n} (p-1)(U)^{p-2}  \widetilde{v}^2 =0.
\label{wxu}
\end{eqnarray}
This is a contradiction, thus proving the proposition.
\end{proof}

Now, we write for $ \phi \in K^{\perp}_{\ve ,x} $,
\begin{equation}
\label{seq}
S_{\ve}(\Ue+\phi)= S_{\ve}(\Ue)+S'_{\ve}(\Ue)\phi+\widetilde{N}_{\ve,x}(\phi),
\end{equation}
where
\begin{eqnarray*}
\widetilde{N}_{\ve,x}(\phi)&=&S_{\ve}(\Ue+\phi)- S_{\ve}(\Ue)-S'_{\ve}(\Ue)\phi \\
&=& -i^*_{\ve}\left( ((\Ue+\phi)^+)^{p-1}-(\Ue)^{p-1}-(p-1)(\Ue)^{p-2}\phi\right).
\end{eqnarray*}
Applying $\Pi^{\perp}_{\ve,x}$ to (\ref{seq}) we see that (\ref{perpeq}) is  equivalent to
\begin{equation}
\label{eql}
L_{\ve,x}(\phi)=N_{\ve,x}(\phi) - \Pi^{\perp}_{\ve,x} (S_{\ve} (U_{\ve,x} )),
\end{equation}
where
\[
N_{\ve,x}(\phi)\doteq - \Pi^{\perp}_{\ve,x} (\widetilde{N}_{\ve,x}(\phi) )   = \Pi^{\perp}_{\ve,x}\Big\{i^*_{\ve}\Big[((\Ue+\phi)^+ )^{p-1}-(\Ue)^{p-1}-(p-1)(\Ue)^{p-2}\phi\Big]\Big\}.
\]

We are now ready to prove the main result of this section. 

\begin{proposition}\label{SecPropoPrinc}
There exists an $\ve_o>0$ and $A>0$ such that for any $x\in M$ and for any $\ve\in (0,\ve_o)$ there exists a unique $\phi_{\ve,x}=\phi(\ve,x) \in K^{\perp}_{\ve ,x} $ that solves Eq. (\ref{perpeq}) with $\|\phi_{\ve,x}\|_{\ve} \leq A$. Moreover, there exists a constant
$c_o >0$ independent of $\ve$ such that 
\[
\|\phi_{\ve,x}\|_{\ve} \leq c_o \ve^2,
\]
and $x\rightarrow \phi_{\ve,x}$ is a $C^2$ map.
\end{proposition}
\begin{proof}
In order to solve Eq. (\ref{perpeq}), or equivalently Eq. (\ref{eql}), we have to find a fixed point of the operator $T_{\ve,x}:K^{\perp}_{\ve,x}\rightarrow K^{\perp}_{\ve,x}$ given by
\[
T_{\ve,x}(\phi)\doteq L^{-1}_{\ve,x}\left(N_{\ve,x}(\phi)- \Pi^{\perp}_{\ve,x} (S_{\ve} (U_{\ve,x} )) \right).
\] 
Now, from Proposition \ref{invl} we have that there is a constant $C>0$ such that
\begin{equation}
\label{T}
\|T_{\ve,x}(\phi)\|_{\ve}\leq C\left(\|N_{\ve,x}(\phi)\|_{\ve}+\| \Pi^{\perp}_{\ve,x} (S_{\ve} (U_{\ve,x} )) \|_{\ve} \right), \quad \forall \phi\in K^{\perp}_{\ve,x}.
\end{equation}

\begin{claim} 
\label{1}
For any $b\in (0,1)$ there exist constants  $a,\ve_o >0$  such that for any
$\ve \in (0, \ve_o )$,   if  $\phi_1 , \phi_2 \in K^{\perp}_{\ve,x}$,  $\| \phi_1 \|_{\ve}$,
 with $\| \phi_2 \|_{\ve} <a, $ then $\| N_{\ve,x} (\phi_1 ) - N_{\ve,x} (\phi_2 ) \|_{\ve} \leq b \| \phi_1 - \phi_2 \|_{\ve}$.
 \end{claim}
\begin{proof}[Proof of Claim \ref{1}]
  \[
    N_{\ve, x}(\phi_1)-N_{\ve, x}(\phi_2)=\Pi^{\perp}\{S_{\ve}(U_{\ve,x}+ \phi_2)-S_{\ve}(U_{\ve,x}+ \phi_1)-S^{'}_{\ve}(U_{\ve,x})(\phi_2 - \phi_1)\}
    \]
    Therefore,
    \[
    \| N_{\ve, x}(\phi_1)-N_{\ve, x}(\phi_2)\|_{\ve}\leq \|S_{\ve}(U_{\ve,x}+ \phi_2)-S_{\ve}(U_{\ve,x}+ \phi_1)-S^{'}_{\ve}(U_{\ve,x})(\phi_2 - \phi_1) \|_{\ve}
    \]
  
\[ =\|   i^*_{\ve} \left( ((U_{\ve,x}+ \phi_1)^+  )^{p-1}   -((U_{\ve,x} + \phi_2 )^+  )^{p-1} +  (p -1) U_{\ve,x}^{p -2}   (\phi_2 - \phi_1)   \right)   \|_{\ve}
\]

    \[
    \leq  c\Big|((U_{\ve,x}+ \phi_1)^+ )^{p-1} -((U_{\ve,x}+ \phi_2)^+ )^{p-1} - (p-1 )(U_{\ve,x})^{p-2} (\phi_1-\phi_2) \Big|_{p',\ve} 
    \]

By the Intermediate Value Theorem, there is a $\lambda \in [0,1]$ such that 
\begin{eqnarray}
\label{mvtl}
|(U_{\ve,x}+ \phi_1)^+ )^{p_{m+n}-1} -((U_{\ve,x}+ \phi_2)^+ )^{p_{m+n}-1}|_{p',\ve} =\\
\nonumber
| (p_{m+n}-1 )(U_{\ve,x} + \phi_1 +
 \lambda (\phi_2 -\phi_1 ))^{p-2} (\phi_2-\phi_1)|_{p',\ve}.
\end{eqnarray}

Then, we have from Eq. \eqref{mvtl} that
\begin{eqnarray*}
 &&\vert  ((U_{\ve,x}+ \phi_1)^+ )^{p-1} -((U_{\ve,x}+ \phi_2)^+ )^{p-1} - (p-1 )(U_{\ve,x})^{p-2} (\phi_1-\phi_2) \vert_{p',\ve} \\
&=& \vert [ (p-1 )(U_{\ve,x} +\phi_1 +  \lambda (\phi_2 -\phi_1 ))^{p-2} -
 (p-1 )(U_{\ve,x})^{p-2} ]  (\phi_1-\phi_2)  \vert_{p',\ve} \\
&\leq & c \vert (U_{\ve,x} +\phi_1 +  \lambda (\phi_2 -\phi_1 ))^{p-2}  - (U_{\ve,x})^{p-2} \vert_{\frac{p}{p-2},\ve} \vert  (\phi_2 - \phi_1)  \vert_{p,\ve}  \\
&\leq& c \vert (U_{\ve,x} + \phi_1 \lambda (\phi_2 -\phi_1 ))^{p-2}  - (U_{\ve,x})^{p-2} \vert_{\frac{p}{p-2},\ve} \|  (\phi_2 - \phi_1)  \|_{\ve }, 
\end{eqnarray*}
by H\"{o}lder's inequality and Remark 2.2. In order to complete the estimate we  need the following elementary observation which appeared in 
\cite[Lemma 2.1]{Li}.
  Let $a >0$ and $b\in \mathbb{R}$, then
\begin{equation}
\label{yy}
\rvert |a+b|^\beta -a^\beta\rvert \leq\begin{cases}
 C(\beta)\min\{|b|^\beta , a^{\beta-1}|b|\}& \text{ if }  0<\beta<1.\\
 C(\beta)(|a|^{\beta-1}|b| +|b|^\beta) & if \beta\geq 1.
\end{cases}
\end{equation}

Applying  (\ref{yy}),  we see that  for all  $v\in H_\ve$ 
\begin{equation}
\label{LYanYan}
 | (U_{\ve,x}+ v)^{p -2} -(U_{\ve,x})^{p-2} | \leq\begin{cases}
 C(p)|v|^{p-2}& \text{ if }  2<p<3 .\\
 C(p)\Big ( |U_{\ve,x}|^{p-3}|v|+|v|^{p-2}\Big ) &\text{ if } p \geq 3.
\end{cases}
\end{equation}
Then, it follows that
\begin{equation}
\label{LYanYan2}
 | (U_{\ve,x}+ v)^{p-2} -(U_{\ve,x})^{p-2} |_{\frac{p}{p-2},\ve} \leq\begin{cases}
 C(p)|v|_{p,\ve}^{p-2}& \text{ if }  2<p<3,\\
 C(p)\Big ( |U_{\ve,x}|^{p-3}_{p,\ve}|v|_{p,\ve}+|v|^{p-2}_{p,\ve}\Big ) &\text{ if } p \geq 3.
\end{cases}
\end{equation}
Using (\ref{LYanYan2}) and Remark 2.2 we can see that  if $a$ is small enough then 
\[  c \Big| (U_{\ve,x} + \lambda (\phi_2 -\phi_1 ))^{p-2}  - (U_{\ve,x})^{p-2} \Big|_{\frac{p}{p-2},\ve} <b,
\]
proving the claim.

\end{proof}

In similar fashion we can prove the following claim.
\begin{claim} 
\label{2}For any $b \in (0,1)$ there exist  constants $a>0$ and $\ve_o >0$    such that for 
any $\ve \in (0,\ve_o )$, if $\| \phi \|_{\ve} <a$ then $\| N_{\ve,x} (\phi ) \|_{\ve} 
\leq b \| \phi \|_{\ve}$.
\end{claim}

\begin{proof}[Proof of Claim \ref{2}]
  \[ \| N_{\ve,x}(\phi)\|_{\ve}= \|\Pi^{\perp}\{S_{\ve}(U_{\ve,x}+ \phi)-S_{\ve}(U_{\ve,x})-S^{'}_{\ve}(U_{\ve,x})(\phi)\}\|_{\ve}\]
  \[=\| i_{\ve}^{*}((U_{\ve,x})^{p_{m+n}-1}-((U_{\ve,x} + \phi)^+ )^{p_{m+n}-1} + (p_{m+n}-1)(U_{\ve,x})^{p_{m+n}-2} \phi)\|_{\ve},
  \]
  and we can apply the Intermediate Value Theorem and Remark 4.3, as in the proof of Claim \ref{1}, to prove the claim.

\end{proof}

Now we prove the first statements of the proposition using the claims.
Let $C$ be the constant in (\ref{T}) and take
$b=\frac{1}{2C}$. Let $a$ be the constant given by Claim \ref{1} and Claim \ref{2} (the minimum of the two, to
be precise). From Lemma \ref{bRs} and Claim \ref{2} there
exists $\ve_o >0$ such that if $\ve \in (0, \ve_o )$ then $T_{\ve,x}$ sends the ball of radius $a$ in $K^{\perp}_{\ve,x}$
to itself.

If $\| \phi_1 \|_{\ve}$,
 $\| \phi_2 \|_{\ve} <a$ , we have that 
\[
\|T_{\ve,x}(\phi_1)-T_{\ve,x}(\phi_2)\|_{\ve}\leq C\|N_{\ve,x}(\phi_1)-N_{\ve,x}(\phi_2)\|_{\ve} \leq \frac{1}{2} \| \phi_1 - \phi_2 \|_{\ve}.
\]

We see then that $T_{\ve,x}$ is a contraction in the ball of radius $a$.
It follows that it has a unique fixed point there. The fixed point is obtained for instance as the limit
ot the sequence $a_k = T_{\ve,x}^k (0)$. Note that $\| a_1 \|_{\ve} \leq C\ve^2 $ by
Lemma \ref{bRs} and then from Claim \ref{1} we have that for all $k$, $\| a_k \|_{\ve} \leq 2C\ve^2$.

It remains to prove that the map $x\rightarrow \phi_{\ve,x}$ is  $C^2$. In order to show this, we apply the Implicit Function Theorem to the $C^2-$function $G:M\times H_{\ve}\rightarrow H_{\ve}$ defined by

\[
G(x,u)=\Pi^{\perp}_{\ve,x}\Big\{S_\ve (U_{\ve,x}+\Pi^{\perp}_{\ve,x}u)\Big\}+\Pi_{\ve,x}u.
\]
Observe that $G(x,\phi_{\ve,x})=0$, and that the derivative $\frac{\pa G}{\pa u}(x,\phi_{\ve,x}):H_{\ve}\rightarrow H_{\ve}$ is given by

\[
\frac{\pa G}{\pa u}(x,\phi_{\ve,x})(u)=\Pi^{\perp}_{\ve,x}\Big\{S_{\ve}^{'} (U_{\ve,x}+\phi_{\ve,x})\Pi^{\perp}_{\ve,x}u\Big\}+\Pi_{\ve,x}u
\]


\vspace{.3cm}

The  proof would be done if we show the next claim. 

\begin{claim}
\label{3}
 For $\ve>0$ small enough, there is $C>0$ such that
\[
\Big\| \frac{\pa G}{\pa u}(x,\phi_{\ve,x})(u)\Big\|_{\ve}\geq C\|u\|_{\ve}, 
\]
for every $x\in M$.
\end{claim}
\begin{proof} [Proof of Claim \ref{3}]
We have that for $c=\frac{1}{\sqrt{2}}$ that
$$\Big\|\frac{\pa G}{\pa u}(x,\phi_{\ve,x})(u)\Big\|_{\ve} \geq c\Big\|\Pi^{\perp}_{\ve,x}\Big\{S'_\ve(U_{\ve,x}  +\phi_{\ve,x})\Pi^{\perp}_{\ve,x}(u)\Big\}\Big\|_\ve+ c\Big\|\Pi_{\ve,x}(u)\Big\|_\ve$$

$$= c\Big\|\Pi^{\perp}_{\ve,x}\Big\{S'_\ve(U_{\ve,x})\Pi^{\perp}_{\ve,x}(u) +  S'_\ve(U_{\ve,x}  +\phi_{\ve,x})\Pi^{\perp}_{\ve,x}(u) - S'_\ve(U_{\ve,x})\Pi^{\perp}_{\ve,x}(u)\Big\}\Big\|_\ve+ c\Big\|\Pi_{\ve,x}(u)\Big\|_\ve$$

$$\geq c\Big\|\Pi_{\ve,x}(u)\Big\|_\ve + c\Big\|L_{\ve,x}(\Pi^{\perp}_{\ve,x}(u))\Big\|_\ve - c\Big\|\Pi^{\perp}_{\ve,x}\Big\{S'_\ve(U_{\ve,x} +\phi_{\ve,x})\Pi^{\perp}_{\ve,x}(u) - S'_\ve(U_{\ve,x})\Pi^{\perp}_{\ve,x}(u)\Big\}\Big\|_\ve $$

It follows from Proposition 4.1 that, for another constant $c>0$, $\Big\|L_{\ve,x}(\Pi^{\perp}_{\ve,x}(u))\Big\|_\ve \geq c \Big\| \Pi^{\perp}_{\ve,x}(u)
\Big\|_{\ve} $. Then we have that for some constant $C>0$, 
$$ c\Big\|\Pi_{\ve,x}(u)\Big\|_\ve + c\Big\|L_{\ve,x}(\Pi^{\perp}_{\ve,x}(u))\Big\|_\ve \geq C \|u\|_{\ve}.$$
Therefore, it only remains to prove that 
$$\lim_{\ve \rightarrow 0} \Big\|\Pi^{\perp}_{\ve,x}\Big\{S'_\ve(U_{\ve,x} +\phi_{\ve,x})\Pi^{\perp}_{\ve,x}(u) - S'_\ve(U_{\ve,x})\Pi^{\perp}_{\ve,x}(u)\Big\}\Big\|_\ve \ =0 .$$
But, 
$$S'_\ve(U_{\ve,x} +\phi_{\ve,x})\Pi^{\perp}_{\ve,x}(u) - S'_\ve(U_{\ve,x})\Pi^{\perp}_{\ve,x}(u) =
 (p-1) i_{\ve}^*
( (U_{\ve,x} +\phi_{\ve,x})^{p-2} - (U_{\ve,x})^{p-2} \Pi^{\perp}_{\ve,x}(u) ) .$$
Hence, as in the proof of Claim \ref{1},  
$$\Big\| S'_\ve(U_{\ve,x} +\phi_{\ve,x})\Pi^{\perp}_{\ve,x}(u) - S'_\ve(U_{\ve,x})\Pi^{\perp}_{\ve,x}(u) \Big\|_\ve \leq 
c | ( (U_{\ve,x} +\phi_{\ve,x})^{p-2} - (U_{\ve,x})^{p-2}) \Pi^{\perp}_{\ve,x}(u) |_{p',\ve}$$

$$\leq c | ( (U_{\ve,x} +\phi_{\ve,x})^{p-2} - (U_{\ve,x})^{p-2}) |_{\frac{p}{p-2},\ve} | \Pi^{\perp}_{\ve,x}(u) |_{p,\ve} $$

$$\leq c  | ( (U_{\ve,x} +\phi_{\ve,x})^{p_{m+n}-2} - (U_{\ve,x})^{p-2}) |_{\frac{p}{p-2},\ve} | \| u \|_{\ve} .$$

Arguing as in the end of the proof of Claim \ref{1} we can see that 
$$\lim_{\ve \rightarrow 0} | ( (U_{\ve,x} +\phi_{\ve,x})^{p-2} - (U_{\ve,x})^{p-2}) |_{\frac{p}{p-2},\ve} =0,$$ 
thus completing the proof of the claim.

\end{proof}

This finishes the proof of the proposition.

\end{proof}

\section{\textbf{Proof of Theorem 1.1}}

Recall that the critical points of  the functional $J_{\varepsilon} : H^1 (M) \rightarrow \R$ given by

$$J_{\varepsilon} (u) = \varepsilon^{-n} \int_M \left( \frac{1}{2} \varepsilon^2 \| \nabla u \|^2  +  \frac{{\bf s}_g  \varepsilon^2 + a_{m+n} }{2a_{m+n}} u^2 -\frac{1}{p}  (u^+ )^{p} \right) d\mu_g  ,$$

\noindent
are the positive solutions of Eq. (\ref{yam-nor}).

Proposition 4.2  tells us that there exists $\ve_o>0$ such that for $\ve \in (0,\ve_o )$ and  $x\in M$ there exists a uniquely defined
$ \phi_{\varepsilon , x} \in K^{\perp}_{\ve ,x} $ such that   $U_{\varepsilon ,x} + \phi_{\varepsilon , x}$ solves Eq. (\ref{perpeq}).  In order to
finish the proof of
Theorem 1.1 we have to establish the following  result.

\begin{proposition}\label{PrincipalProp} There exists $\ve_o >0$ such that if $\ve \in (0, \ve_o )$
and $x_o\in M$ is a critical point of $F_{\varepsilon} : M \rightarrow \R $, where 
\begin{equation}
F_{\varepsilon} (x) \doteq J_{\varepsilon} (U_{\varepsilon ,x} + \phi_{\varepsilon , x} ) ,
\end{equation}
 then
$U_{\varepsilon ,x_o } + \phi_{\varepsilon , x_o } $ is a positive solution of Eq. (\ref{yam-nor}).
\end{proposition}

\begin{proof} Let   $x_o\in M$ be a critical point of $F_{\ve}$ where $\ve>0$. We need to  
show  that for each $\varphi \in H_{\ve} (M)$ one has
that 

$$\langle S_{\ve}  (  U_{\varepsilon ,x_o } + \phi_{\varepsilon , x_o })    , \varphi \rangle_{\ve}  =0 .$$
If $\varphi \in K^{\perp}_{\varepsilon ,x_o  }$ then

$$\langle S_{\ve} \ ( U_{\varepsilon ,x_o } + \phi_{\varepsilon , x_o })  , \varphi \rangle_{\ve}  = \langle \Pi^{\perp}_{\ve,x} (S_{\varepsilon} ( U_{\varepsilon ,x_o } + \phi_{\varepsilon , x_o })  ),\varphi \rangle_{\ve} =0 ,$$
 since $U_{\varepsilon ,x_o } + \phi_{\varepsilon , x_o}$ solves Eq. (\ref{perpeq}).

Then  it is enough to show that $\langle  S_{\ve}  ( U_{\varepsilon ,x_o } + \phi_{\varepsilon , x_o }) , \varphi \rangle_{\ve} =0$ if $\varphi \in K_{\varepsilon, x_o }$. 
On the other hand we know that 
$\langle S_{\ve} ( U_{\varepsilon ,x_o } + \phi_{\varepsilon , x_o })  , \varphi \rangle_{\ve} =0$ if $\varphi$ is tangent to 
the map $x \mapsto V(x)= U_  {\varepsilon ,x} + \phi_{\varepsilon , x } $ at $x_o$. And since $M$ and $K_{\ve,x_o}$ have the same dimension it is
enough to see that the projection $\Pi_{{\varepsilon, x_o}}  \circ  D_{x_o} V : T_{x_o} M \rightarrow K_{\varepsilon , x_o}$ is injective. 

Then to finish the proof it is enough to show that,
fixing geodesic coordinates centered at $x_o$, for any $v \in \R^n $

\begin{equation}\label{AAAA}
\Big\langle \frac{\partial}{\partial v}  (U_{\varepsilon ,x} + \phi_{\varepsilon , x} )  (x_o ), W_{\ve ,x_o }^v  \Big\rangle_{\ve} \neq 0 .
\end{equation}

Note that $\langle \phi_{\varepsilon , x}  , W_{\ve ,x  }^v  \rangle_{\ve} =0$. Then

\[
\Big\langle \frac{\partial}{\partial v}  ( \phi_{\varepsilon , x} ) , W_{\ve ,x_o }^v  \Big\rangle_{\ve} = - 
\Big\langle  \phi_{\varepsilon , x}  , \frac{\partial}{\partial v}  W_{\ve ,x_o }^v  \Big\rangle_{\ve}  .
\]
As we pointed out in  (\ref{end1}),  we have
$$\lim_{\ve \rightarrow 0} \ \ve^2 \  \| \frac{\partial}{\partial v}  W_{\ve ,x_o }^v \|_{\ve}   =0 .$$
Then, it follows from Cauchy-Schwarz inequality and Proposition \ref{SecPropoPrinc} that 
$$\lim_{\ve \rightarrow 0} \Big\langle \frac{\partial}{\partial v}  ( \phi_{\varepsilon , x} ) , W_{\ve ,x_o }^v  \Big\rangle_{\ve} =0.$$
From  (\ref{end2}),
\[
\lim_{\ve \rightarrow 0} \ve 
\Big\langle \frac{\partial}{\partial v}  (U_{\varepsilon ,x}  ) , W_{\ve ,x_o }^v  \Big\rangle_{\ve} = \langle \psi^v , \psi^v \rangle >0.
\]
Then, for $\ve >0$ small enough (\ref{AAAA}) holds, and the proposition is proved.

\end{proof}

\section{Analytic proof that $\beta_{2,2} \neq 0$}

In \cite{Rey_Ruiz} C. Rey and M. Ruiz numericallly checked   that $\beta_{m,n} <0$  if $n+m \leq 9.$ In this section we prove that $\beta_{m,n}$ is not equal to zero  for values $m$ and $n$ such that $m+n=4$. Fix $m ,n$ and let
\[
\beta = a_{m+n} \  \beta_{m,n} = \int_{\R^n}U^2-\frac{a_{m+n}}{n(n+2)}\int_{\R^n}|\nabla U|^2|z|^2dz.
\]
Recall that $p=p_{m+n}$. 
\begin{theorem}
  If  $m$ and $n$ such that $m+n=4$, then $\beta < 0.$ If $n\neq 4$, $m+n >4$ we have
  \[\beta=\frac{6-a_{m+n}}{n(n-4)}\int_{\R^n}\left(\frac{2}{p}\cdot\frac{m}{(m+n-4)}U^p-U^2\right)|z|^2dz.\]
  \end{theorem}
\begin{proof}
We know that $U$ satisfies
\begin{equation}
\label{equ}
\Delta U=U-U^{p-1}.
\end{equation}
Let us multiply \eqref{equ} by $U|z|^2$ and integrate
\begin{eqnarray*}
\int_{\R^n}\left(U^2-U^p\right)|z|^2dz &=& \int_{\R^n}\Delta U\cdot U|z|^2dz\\
&=&-\int_{\R^n}\langle\nabla U, \nabla \left(U\cdot |z|^2\right)\rangle dz\quad \text{(by the Divergence Theorem)}\\
&=&-\int_{\R^n}|\nabla U|^2|z|^2dz-2\int_{\R^n}U\langle \nabla U,z\rangle \\
&=&-\int_{\R^n}|\nabla U|^2|z|^2dz-\int_{\R^n}\langle \nabla U^2,z\rangle \\
&=&-\int_{\R^n}|\nabla U|^2|z|^2dz+n\int_{\R^n}U^2dz\\
\end{eqnarray*}
Hence,
\begin{equation}
\label{equ2}
n\int_{\R^n}U^2dz=\int_{\R^n}|\nabla U|^2|z|^2dz+\int_{\R^n}\left(U^2-U^p\right)|z|^2dz.
\end{equation}

It is proved in  Lemma 5.5 in \cite{Micheletti-Pistoia} that 

\[
\int_{\R^n} \left(\frac{\partial U}{\partial z_i}\right)^2z^2_idz= \frac{1}{2}
\int_{\R^n}|\nabla U|^2 z_i^2 dz+\int_{\R^n}\left( (1/2) U^2- (1/p) U^p\right) z_i^2dz
\]
\[
=\frac{1}{2n}
\int_{\R^n}|\nabla U|^2 |z|^2 dz+\int_{\R^n}\left( (1/2n) U^2- (1/np) U^p\right) |z|^2dz
\]

And using that (see for instance the proof of Lemma 3.3 in \cite{Rey_Ruiz})
\[
\int_{\R^n} \left(\frac{\partial U}{\partial z_i}\right)^2z^2_idz=\int_{\R^n}\left(\frac{U'(|z|)}{|z|}\right)^2z^4_idz=\frac{3}{n(n+2)}\int_{\R^n}|\nabla U|^2|z|^2dz,\quad i=1,\ldots,n,
\]

\noindent
we have
\begin{eqnarray}
\label{equp}
\left(\frac{n-4}{n+2}\right)\int_{\R^n}|\nabla U|^2|z|^2dz
&=&\frac{2}{p}\int_{\R^n}U^p|z|^2dz-\int_{\R^n}U^2|z|^2dz.
\end{eqnarray}

Now,  observe that by (\ref{equ2})
\begin{equation}
\label{beta}
n\beta= \left(\frac{n+2-a_{m+n}}{n+2}\right)\int_{\R^n}|\nabla U|^2|z|^2dz+\int_{\R^n}\left(U^2-U^p\right)|z|^2dz.
\end{equation}
Hence, by \eqref{equp} and \eqref{beta}
\begin{eqnarray*}
n(n-4)\beta&=&\frac{1}{p}\cdot\left(2n+4-2a_{m+n}+(4-n)p\right)\int_{\R^n}U^p|z|^2dz-(6-a_{m+n})\int_{\R^n}U^2|z|^2dz\\
&=&\frac{4}{p}\cdot\frac{m}{m+n-2}\int_{\R^n}U^p|z|^2dz-(6-a_{m+n})\int_{\R^n}U^2|z|^2dz.
\end{eqnarray*}
Notice that if $n=3,m=1$ or $n=m=2$, we have $a_{m+n}=6$. Therefore, in these two cases we obtain $\beta < 0$.

Finally, if $n\neq 4$ and $m+n>4$,
\[
\beta= \frac{6-a_{m+n}}{n(n-4)}\int_{\R^n}\left(\frac{2}{p}\cdot\frac{m}{(m+n-4)}U^p-U^2\right)|z|^2dz.
\]

\end{proof}

\textbf{Acknowledgments.} The authors wish to thank Prof. Jimmy Petean for his constant interest and the many helpful conversations on the Yamabe equation.

\end{document}